\documentclass[twoside,a4paper,12pt,centertags]{amsart}
\usepackage{amsmath,amssymb,verbatim,vmargin}
\usepackage[bookmarks=true]{hyperref}   
\usepackage{color}
\usepackage{tikz}
\usetikzlibrary{arrows, decorations.markings}

\theoremstyle{plain}
\newtheorem{thm}{Theorem}[section]

\newtheorem{cor}[thm]{Corollary}
\newtheorem{prop}[thm]{Proposition}

\newtheorem{alg}[thm]{Algorithm}

\theoremstyle{definition}

\newtheorem{ex}[thm]{Example}

\mathchardef\semic="303B
\newcommand{\dirac}{{\mathbf D}}

\newcommand{\wedg}{\mathbin{\scriptstyle{\wedge}}}
\newcommand{\lctr}{\mathbin{\lrcorner}}
\newcommand{\cc}[1]{{#1}^{\text c}}
\newcommand{\inv}[1]{{\widehat{#1}}}
\newcommand{\rev}[1]{\overline{#1}}
\newcommand{\R}{{\mathbf R}}
\newcommand{\C}{{\mathbf C}}

\newcommand{\mH}{{\mathcal H}}

\newcommand{\mF}{{\mathcal F}}

\DeclareMathOperator{\re}{Re}
\newcommand{\im}{\text{{\rm Im}}\,}
\newcommand{\sett}[2]{ \{ #1 \, \semic \, #2 \} }

\newcommand{\supp}{\text{{\rm supp}}\,}

\newcommand{\ran}{\textsf{R}}

\newcommand{\clos}[1]{\overline{#1}}

\newcommand{\barint}{\mbox{$ave \int$}}

\newcommand{\pv}{\text{{\rm p.v.\!}}}
\newcommand{\ev}{\text{ev}}
\newcommand{\od}{\text{od}}

\newcommand{\inc}{\text{inc}}

\def\barint_#1{\mathchoice
            {\mathop{\vrule width 6pt
height 3 pt depth -2.5pt
                    \kern -8.8pt
\intop}\nolimits_{#1}}%
            {\mathop{\vrule width 5pt height
3 pt depth -2.6pt
                    \kern -6.5pt
\intop}\nolimits_{#1}}%
            {\mathop{\vrule width 5pt height
3 pt depth -2.6pt
                    \kern -6pt
\intop}\nolimits_{#1}}%
            {\mathop{\vrule width 5pt height
3 pt depth -2.6pt
          \kern -6pt \intop}\nolimits_{#1}}}

\usepackage{color}

\definecolor{gr}{rgb}   {0.,   0.8,   0. }
\definecolor{bl}{rgb}   {0.,   0.5,   1. }
\definecolor{mg}{rgb}   {0.7,  0.,    0.7}

\makeatletter
\@namedef{subjclassname@2010}{%
  \textup{2010} Mathematics Subject Classification}
\makeatother

\begin{document}

\title[Boosting the Maxwell double layer potential]
{Boosting the Maxwell double layer potential using a right spin factor}
\author[Andreas Ros\'en]{Andreas Ros\'en$\,^1$}
\thanks{$^1\,$Formerly Andreas Axelsson}
\address{Andreas Ros\'en\\Mathematical Sciences, Chalmers University of Technology and University of Gothenburg\\
SE-412 96 G{\"o}teborg, Sweden}
\email{andreas.rosen@chalmers.se}
\dedicatory{Dedicated to the memory of Alan G. R. McIntosh}

\begin{abstract}
We construct new spin singular integral equations for solving scattering problems for Maxwell's equations, both against perfect conductors
and in media with piecewise constant permittivity, permeability and conductivity, improving and extending earlier formulations by the author.
These differ in a fundamental way from classical integral equations, which use double layer potential operators,
and have the advantage of having a better condition number, in particular in Fredholm sense and on Lipschitz regular interfaces, and do not suffer
from spurious resonances.
The construction of the integral equations builds on the observation that the double layer potential factorises into a boundary
value problem and an ansatz. We modify the ansatz, inspired by a non-selfadjoint local elliptic boundary condition
for Dirac equations. 
\end{abstract}

\keywords{Maxwell scattering, integral equation, spurious resonances}
\subjclass[2010]{45E05, 78M15, 15A66}

\maketitle

%
%
%
\section{Introduction}

%
%
The classical principal value double layer potential is the operator 
$$
   Kf(x)= 2 \pv \int_{\partial\Omega} (\nabla\Phi)(y-x)\cdot \nu(y) f(y) d\sigma(y), \qquad x\in\partial\Omega,
$$
acting on functions $f$ the boundary of a bounded Lipschitz domain $\Omega\subset\R^n$, and
using the Laplace fundamental solution $\Phi$ and the outward pointing unit normal vector field $\nu$ for its kernel.
Here $\sigma$ is standard surface measure and we choose to normalize by a factor $2$.
The method of boundary integral equations for solving the classical Dirichlet and Neumann boundary value problems for 
the Laplace equation on $\Omega$ is as follows.
To solve the Dirichlet problem with datum $g_1$, we solve 
$$
  f_1(x) +(Kf_1)(x)   = 2g_1(x), \qquad x\in\partial\Omega, 
$$
from which the harmonic function $u$ is obtained as the double layer potential with density $f_1$.
Similarly, to solve the Neumann problem with datum $g_2$, we solve 
$$
 - f_2(x) +(K^*f_2)(x)    = 2g_2(x), \qquad x\in\partial\Omega,
$$
from which the harmonic function $u$ is obtained as the single layer potential with density $f_2$.

For smooth domains, $K$ is a weakly singular integral, which gives a compact operator on many function space
and invertibility can be deduced by classical Fredholm theory.
For Lipschitz domains, $K$ is a singular integral operator (modulo the factor $\nu$), and its $L_p$ boundedness, $1<p<\infty$,
follows from the seminal work by Coifman, McIntosh and Meyer~\cite{CMcM}. 
For the rest of this paper, we will restrict attention to the most fundamental function space for singular integrals: $L_2$.
On a strongly Lipschitz domain, that is when $\partial\Omega$ is locally a Lipschitz graph, Rellich identities replaces the Fredholm arguments to
show that $I\pm K$ is a Fredholm operator on $L_2(\partial\Omega)$.

%
%

In this paper we derive a new integral equation for solving the Maxwell scattering problem again perfect conductors.
This makes use of Clifford algebra and an embedding of Maxwell's equations into a Dirac equation. 
To explain our ideas and results, we discuss in this introduction the double layer potential in the complex plane where all the main ideas are present but the algebra is simpler:
Clifford algebra simplifies to complex algebra, Maxwell's equations simplify to Cauchy--Riemann's equations, and Dirac solutions simplify
to analytic and anti-analytic functions.
To write $K$ in complex notation when dimension is $n=2$, we replace points $x$ and $y$ by complex numbers $z$ and $w$.
 In the integrand, $\nabla\Phi(y-x)$ becomes $(2\pi (\rev w-\rev z))^{-1}$ and $n(y)d\sigma(y)$ becomes $-idw$.
 With the expression $z\cdot w= \re(\rev zw)$ for the real inner product,
we obtain
\begin{equation}   \label{eq:K2D}
   Kf(z)= \re\left(\frac{1}{\pi i}\pv\int_{\partial\Omega} \frac{f(w)}{w-z}dw \right), \qquad z\in\partial\Omega.
\end{equation}
 Here we recognize the Cauchy integral from complex analysis, and we define the principal value Cauchy integral
 \begin{equation}   \label{eq:classyCauchy}
   Ef(z)= \frac{1}{\pi i}\pv\int_{\partial\Omega} \frac{f(w)}{w-z}dw , \qquad z\in\partial\Omega.
\end{equation}
Superficially, $E$ looks not much different from $K$. Clearly, boundedness of $E$ on a given curve implies boundedness
of $K$ on that curve. However, we are more concerned with invertibility of $I\pm K$, and in this case $E$ is a much simpler object
than $K$.
In fact
$$
    E^\pm = \tfrac 12(I\pm E)
 $$
are projections, although in general not orthogonal.
Here $E^+$ projects onto the interior Hardy space: the subspace consisting of traces of analytic functions in $\Omega$.
The null space of $E^+$, the subspace along which it projects, is the exterior Hardy space consisting of traces of analytic functions
in $\Omega^-$ which vanishes at $\infty$.
This explains the structure of $E$ itself. It is a reflection operator since $E=E^+-E^-$, which reflects the exterior Hardy space $E^-L_2=\ran(E^-)$
across the interior Hardy space $E^+L_2=\ran(E^+)$.
In particular we see that the spectrum is
$\sigma(E)= \{+1, -1\}$.

Hiding behind $K$ there is also a second reflection operator, namely pointwise complex conjugation
$$
  Nf(z)= \rev{f(z)},\qquad z\in\partial\Omega,
$$
which comes along with its two spectral projections $N^+f=\re f$ and $N^-f=\im f$.
Note that although we use complex numbers, we will still regard them mainly as vectors. In particular, we consider the
two basic reflecion operators $E$ and $N$ as real linear operators.
In terms of these projections, the operator $\tfrac 12(I+K)$ used for solving the Dirichlet problem, 
is the composition of two 
(restrictions of) projections, namely 
the interior Cauchy projection 
  \begin{equation}    \label{eqfactor1}
  E^+: N^+L_2 \to E^+L_2
  \end{equation}
  restricted to the subspace of real functions, and
   the real projection 
    \begin{equation}    \label{eqfactor2}
    N^+: E^+L_2 \to N^+L_2
    \end{equation}
  restricted to the interior Hardy subspace.
  This is readily seen from \eqref{eq:K2D}.
Equivalently, $\tfrac 12(1+K)$ is the compression $N^+E^+N^+$ of the Cauchy projection $E^+$ to the subspace $N^+L_2$ 
of real functions.

\begin{figure}
\centering
\begin{tikzpicture}[scale=2.2]
\tikzset{>=latex}
   \draw [ultra thick] (-1.5,0) -- (4,0) node[right] {$N^+L_2$};
   \draw [ultra thick] (0,-1.5) -- (0,3) node[right] {$N^-L_2$};
   \draw [ultra thick, dashed] (-1.26,-0.84) -- (3.33,2.22) node[right] {$E^+L_2$};
   \draw [ultra thick, dashed] (-1.41,0.47) -- (3.78,-1.26) node[right] {$E^-L_2$};
   \draw [->, scale=1.2, dashed] (2.97,0) -- (0.99,0.66);
   \node at (1.9,0.75) {ansatz};
   \draw [->, scale=1.2] (0.99, 0.66) -- (0.99, 0);
   \node at (1.41, 0.25) {bvp};
   \node at (3.6,-0.15) {$f$};
   \node at (1.3,-0.15) {$\tfrac 12(I+K)f$};
   \node at (1,0.9) {$E^+f$};
\end{tikzpicture}
\caption{Factorization of the operator  $\tfrac 12(I+K)$ into an ansatz (restriction of $E^+$) and the boundary value problem (bvp, restriction of $N^+$).}   \label{fig:alan}
\end{figure}
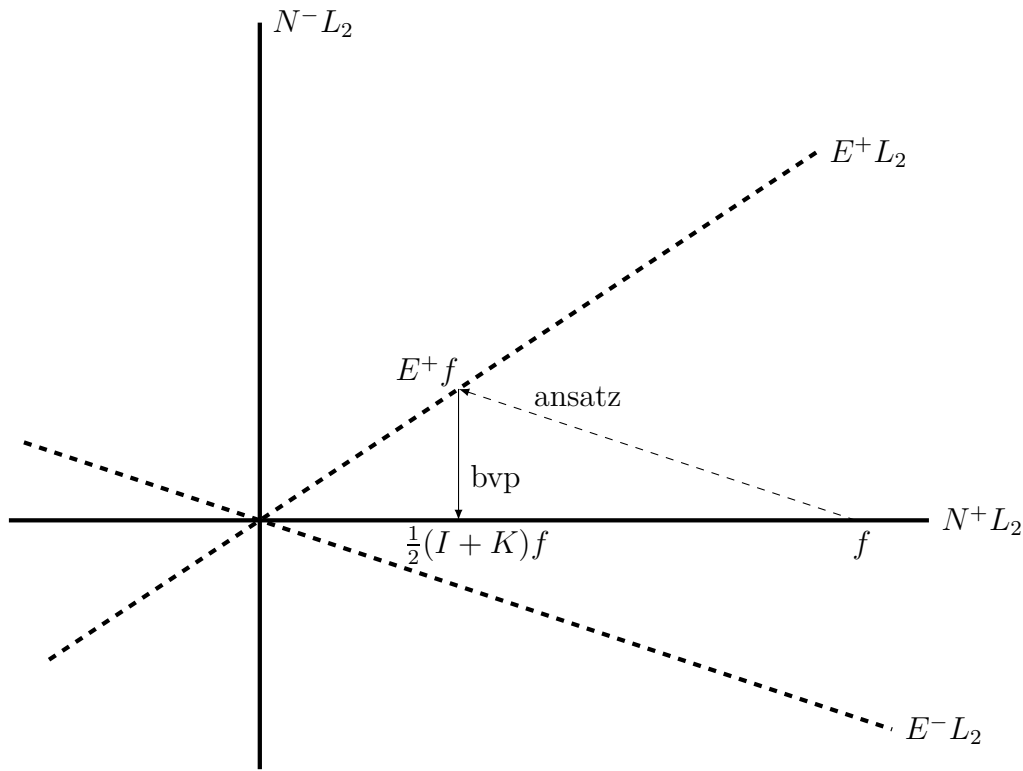 

The factorization into \eqref{eqfactor1} and \eqref{eqfactor2} explains the relation between the boundary value problem and $K$. 
However, we need to re-interpret the Dirichlet problem as a boundary value problem for analytic functions: We regard harmonic functions as real parts of analytic function, neglecting some possible minor topological obstructions.
Having switched in this way from the Laplace equation to the Cauchy--Riemann system, the Dirichlet problem now amounts 
to the Hilbert problem
of finding the analytic function in $\Omega$ which has a prescibed real part at the boundary.
Thus, in terms of operators, solving the boundary value problem means inverting the map  \eqref{eqfactor2}.

As for the right factor \eqref{eqfactor1}, this should be viewed as an ansatz for the solution (or more precisely the trace of the analytic function).
The {\em reason d'\^etre} for  \eqref{eqfactor1} is that it replaces the non-locally defined Hardy space $E^+L_2$ by the  locally defined space $N^+L_2$,
which of course is crucial for numerical computations.
Ideally, we would like \eqref{eqfactor1} to map  $N^+L_2$ bijectively onto $E^+L_2$, with a not to large condition number.
Recall that the condition number 
$$\kappa(T)=\|T\|\|T^{-1}\|$$
of a linear operator
 is a measure of how easy $T$ is to invert numerically.
Unfortunately, the ansatz \eqref{eqfactor1} may not be injective. Moreover, we show in Section~\ref{sec:2d} that even the Fredholm condition number 
 \eqref{eqfactor1} is comparable to $1/\theta$ as $\theta\to 0$, when $\partial\Omega$ contains a corner with angle $\theta$.
 This demontrates that numerically, the integral equation $(I+K)h=g$ may be much worse than the boundary value problem $N^+ h=g$, $h\in E^+L_2$,
 that it is used for solving.
 
 In this paper, we propose a new integral equation for solving the same boundary value problem, which arise upon 
 modifying the maps \eqref{eqfactor1} and \eqref{eqfactor2} to 
   \begin{equation}    \label{eqfactor1tilde}
    E^+:S^-\tilde L_2 \to E^+\tilde L_2
  \end{equation}
  and
    \begin{equation}    \label{eqfactor2tilde}
    N^+: E^+\tilde L_2 \to  N^+\tilde L_2
    \end{equation}
respectively.
Our motivation is a non-selfadjoint local elliptic boundary value problem for Dirac equations
 \begin{equation}   \label{eq:spindiracbvp}
   \begin{cases}
      \dirac f=0,&\qquad \text{in }\Omega,\\
      (1-\nu ) f=g,&\qquad \text{on }\partial\Omega.
   \end{cases}
 \end{equation}
 Details concerning the Dirac operator $\dirac$ is found in Section~\ref{sec:higheralg}.
 At the boundary $S^-:f\mapsto \tfrac 12(1-\nu)f$ acts from the left by Clifford multiplication and yields a projection. 
  We follow conventions from physics that Clifford squares of real vectors are positive. Stokes' theorem shows that
  the eigenvalue problem $\dirac f=ikf$, $\nu f=f$ only give spectrum in $\im k<0$. 
 This gives an indication of that the boundary value problem \eqref{eq:spindiracbvp} has good properties. Indeed \eqref{eq:spindiracbvp} is 
 as well posed as one can possibly hope for when studying time-harmonic waves with wave number $\im k\ge 0$. See Section~\ref{sec:spin} for more details.
 Replacing $N$ by the reflection operator $S:f\mapsto \nu f$ above, yields an ansatz \eqref{eqfactor1tilde} with better properties than
 \eqref{eqfactor1}. Due to its origin, we will refer to this new ansatz as the {\em spin ansatz}.
 However, to be able to use the spin ansatz, we need to embed our differential equation into a Dirac equation.
 This means that we extend the operators $N$ and $E$ to act in a larger space $\tilde L_2$ of multivector valued functions, 
 in which the map  \eqref{eqfactor2} encoding the  original boundary value problem is the restriction to an invariant subspace 
 of a larger map  \eqref{eqfactor2tilde} encoding a Dirac boundary value problem.
 In two dimensions, the Dirac equation amounts to a pair of functions in $\Omega$, one which is analytic and one which is anti-analytic.
 The resulting equation is as follows.
Given Dirichlet datum $g:\partial\Omega\to\R$, we solve the real linear singular integral equation
\begin{multline}   \label{eq:2dspineq}
  h(z)+ \re\left( \frac 1{\pi i} \pv\int_{\partial\Omega} \frac {h(w)}{w-z} dw\right) \\
  +\frac i\pi\pv\int_{\partial\Omega} \re\left( \frac {\rev{t(z)}}{\rev w-\rev z} h(w) \right) |dw|=2g(z), \qquad z\in\Sigma,
\end{multline}
for $h:\partial\Omega\to\C$, where $|dw|$ and $t(z)$ denote the scalar measure and unit tangent vector along $\partial\Omega$ respectively.
The solution $u+iv$ in $\Omega$, $u$ solving the Dirichlet problem and $v$ its harmonic conjugate function, is then
the Cauchy integral of $h$.
The derivation of this spin integral formulation for the Dirichlet problem is found in Example~\ref{ex:2dint}.

For the idea above to work, it is important that the Dirac boundary value 
problem \eqref{eqfactor2tilde} is well posed just like the original boundary value problem \eqref{eqfactor2} which we embed.
In the static case $k=0$ in two dimensions above, both the original boundary value problem and the Dirac boundary 
value problem, are in general well posed only in Fredholm sense.
For time-harmonic  Maxwell's equations however, the wave number is $k\ne 0$, and in this case we will have a well posed boundary value problem
with a unique solution  in a connected exterior domain.
We embed Maxwell's equations into a larger Dirac equation, and also the Dirac boundary value problem will be well posed.
Proceeding similar to above in three dimensions, and with $k\ne 0$, we show in  Example~\ref{ex:maxwellintequ} how the Maxwell
scattering problem in $\R^3\setminus \clos{\Omega}$ against a perfect conductor $\Omega$, can be solved by a singular integral equation
of the form
\begin{equation}   \label{eq:maxwellform}
  \tfrac 14 h(x)+ M(x)\pv \int_{\partial \Omega} \Psi_k(y-x) (1+\nu(y)) h(y) d\sigma(y)= g(x),\qquad x\in\partial\Omega,
\end{equation}
where $M$ is a multiplier involving $\nu$ and where $\Psi_k=\nabla\Phi$ modulo weakly singular terms.
The auxiliary function $h$ on $\partial\Omega$ in this case has four scalar components, but there are no algebraic, differential or integral
 constraints imposed on $h$.
It follows from Section~\ref{sec:spin} that this equation is uniquely solvable with a Fredholm condition number (that is using the Calkin algebra norm 
for operators) comparable to that for the
Maxwell boundary value problem.

In Section~\ref{sec:dielectric}  we generalise \eqref{eq:maxwellform} beyond the case of a perfect conductor, and formulate 
an integral equation with a spin ansatz for solving the Maxwell scattering problem against a finite number of objects with different scalar och constant
permittivity, permeability and conductivity.

%
%
As is often the case in research, both before and after Newton, one is sitting on the shoulders of giants. In my case, many of the ideas 
underlying this paper is a heritage from my PhD supervisor and colleague Alan McIntosh. 
As already mentioned, the $L_2$ boundedness of singular integrals of the kind employed in this paper on Lipschitz surfaces follows from the work of 
Coifman, McIntosh and Meyer~\cite{CMcM}, by Calder\'on's method of rotation from the one dimensional case.
A direct proof in $\R^n$ using Clifford algebra was given by McIntosh~\cite{Mc1}.
In my thesis \cite{Ax, Ax1, Ax2, AMc, Ax3}, I elaborated on the ideas of McIntosh to solve Maxwell's equations by embedding into
the elliptic Dirac equation, and this build on the earlier works of McIntosh and Mitrea~\cite{McMi} and Grognard, Hogan and McIntosh~\cite{AGHMc}.
Although the idea is older than so, the understanding of a boundary value problem in terms of subspaces like in Figure~\ref{fig:alan} is for me a heritage 
from McIntosh.  
In the study of smooth boundary value problems for Dirac operators, see for example the book \cite{BBW} by Boo\ss--Bavnbek and Wojciechowski, it is standard
to formulate boundary conditions in terms of subspaces. But less so in the study of non-smooth boundary value problems for Maxwell's equations
or second order elliptic equations.

It is surprising that it is still a somewhat open problem to find a numerically well behaved boundary integral equation for solving scattering problems for Maxwell's equation.
See Epstein and Greengard  \cite{EG} and Epstein, Greengard and O'Neil \cite{EGoN} for recent new Debye formulations, and 
Colton and Kress~\cite{CK0, CK} for the classical formulations.
The spin integral equations that we propose in the present paper are based on the McIntosh singular integral approach with Clifford algebra from \cite{McMi, AGHMc, Ax}.
However, the integral formulations there are not suitable for numerical computations, since they suffer from spurious resonances and 
the same problems as the classical double layer potential equation.
The work in the present paper began in \cite{Rspin1}, where an equation in the spirit of \eqref{eq:maxwellform} was formulated for solving the Maxwell scattering problem against a perfect conductor, using the formalism from \cite{Ax1}.
Both \eqref{eq:maxwellform} and the spin integral equation from \cite{Rspin1} have the advantages of not suffering from spurious resonances and having an improved condition number, at least in the Fredholm sense for Lipschitz boundaries, compared to classical formulations. However, the equation in  \cite{Rspin1} is for eight unknown scalar functions,
as compared to the four unknown scalar functions for the equation in this paper. 
In both cases, the main novelty lies is the use of an auxiliary spin boundary condition \ref{eq:spindiracbvp}, to obtain a singular integral operator with improved condition number.
To our knowledge, this local non-selfadjoint Dirac boundary condition has not been exploited in this way before, with numerical computations 
in mind.

\section{Higher dimension algebra}  \label{sec:higheralg}

In this section, we fix notation and survey the higher dimensional algebra which we need for
Dirac equations.
See \cite{Ax, Rspin1} for more details.

In particular we recall in Example~\ref{ex:maxwell} how Maxwell's equations fit into this framework. 
We denote by $\Omega^+$ a bounded domain in $\R^n$ with a strongly Lipschitz boundary $\Sigma=\partial\Omega^+$.
This means that locally around each point on the boundary, $\Sigma$ coincides with the graph of a Lipschitz regular function,
suitably rotated. The unbounded exterior domain we denote by $\Omega^-= \R^n\setminus\clos{\Omega^+}$.
Sometimes we abbreviate $\Omega=\Omega^+$. The unit normal vector field on $\Sigma$ pointing into $\Omega^-$ we denote by $\nu$.
The $\R^n$ standard basis is denoted $\{e_j\}_{j=1}^n$.

Functions to be used are defined on subsets of $\R^n$, which take values in the complex exterior algebra for $\R^n$,
which we denote by $\wedge \C^n$. This is the $2^n$-dimensional complex linear space spanned by basis multivectors
$$
  e_{s_1}\wedg e_{s_2}\wedg \ldots\wedg e_{s_j}, \qquad 1\le s_1<s_2<\ldots<s_j\le n,\, 0\le j\le n,
$$
The scalars $\C=\wedge^0\C^n\subset \wedge\C^n$ is the one-dimensional subspace corresponding to $j=0$,
and  the vectors $\C^n=\wedge^1\C^n\subset \wedge\C^n$ is the n-dimensional subspace corresponding to $j=1$.
General objects in the exterior algebra we refer to as multivectors.
For a given $0\le j\le n$, we denote the subspace of $j$-vectors by $\wedge^j\C^n$, and 
 $\wedge^\ev\C^n:= \oplus_j \wedge^{2j}\C^n$ and  $\wedge^\od\C^n:= \oplus_j \wedge^{2j+1}\C^n$
denotes the even subalgebra and odd subspace of the exterior algebra.

We use the hermitean inner product $(\cdot,\cdot)$ on $\wedge\C^n$ for which the above basis multivectors is an ON-basis.
The function space on $\Sigma$ where we consider our integral equations, is the space $L_2(\Sigma)= L_2(\Sigma;\wedge\C^n)$
of square integrable functions $f:\Sigma\to\wedge\C^n$ with inner product
$(f,g)= \int_\Sigma (f(x),g(x)) d\sigma(x)$, where $d\sigma$ denotes standard surface measure.

On $\wedge\C^n$ we use three complex bilinear products: The exterior product $u\wedg w$, the (left) interior product $u\lctr w$,
and the Clifford product $uw$. In the special case when the left factor is a vector $u\in\wedge^1\C^n$, these products are given by
the following basis formulas.
For $s_1<s_2<\ldots <s_j$, let $s=\{s_1,\ldots, s_j\}$ and $e_s=e_{s_1}\wedg e_{s_2}\wedg \ldots\wedg e_{s_j}$, 
and denote by $\epsilon(i,s)$ the number of indices $s_k\in s$ which are $< i$. Then
\begin{align*}
   e_i e_s= e_i\wedg e_s &= (-1)^{\epsilon(i,s)} e_{s\cup \{i\}}, \\
   e_i\lctr e_s &=0,
\end{align*}
when $i\notin s$, and if $i\in s$ then
\begin{align*}
   e_i\wedg e_s &=0, \\
   e_i e_s= e_i\lctr e_s&= (-1)^{\epsilon(i,s)} e_{s\setminus \{i\}}.
\end{align*}
The exterior and Clifford products are the associative complex algebra products with $1=e_\emptyset\in \wedge^0\C^n$ as identity, 
which are uniquely determined by these basis formulas.
The definition of the interior product, which is not associative, for two general multivectors is by duality: We require 
$$
  (u\lctr v, w)= (v, u\wedg w),\qquad u,v,w\in\wedge\C^n,
$$
whenever $u$ has real coordinates.
  Two useful unitary operations on $\wedge\C^n$, which are automorphisms with respect to all three products above, are
 the involution $\inv w$ and reversion $\rev w$, given by
 \begin{align}
   \inv w & = (-e_{s_1})\wedg (-e_{s_2})\wedg \ldots \wedg (-e_{s_j})= (-1)^{|s|}e_s,\\
    \rev w&= e_{s_j}\wedg e_{s_{j-1}}\wedg \ldots\wedg e_{s_1}=(-1)^{|s|(|s|-1)/2}e_s,   \label{eq:reversion}
\end{align}
for $w=e_s= e_{s_1}\wedg e_{s_2}\wedg \ldots\wedg e_{s_j}$. We also make use of coordinate-wise complex conjugation
which we denote by $\cc w$.
For $a\in\wedge^1 \C^n$ and $w\in\wedge\C^n$, the Clifford, interior and exterior products are related by
$$
   aw= a\lctr w+a\wedg w
 $$ 
 and conversely by the Riesz formulas
 \begin{align}
    a\wedg w=\tfrac 12( aw+\inv w a),   \label{eq:rieszwedg}\\
    a\lctr w=\tfrac 12( aw-\inv w a).  \label{eq:rieszlctr}
 \end{align}

Similar to the discussion in the introduction, we now aim to define two pairs of complementary $L_2$ projections, using the 
above algebra.
We define
\begin{align*}
   N^+ f(x) &= \nu(x)\lctr (\nu(x)\wedg f(x)),\qquad x\in\Sigma,\\
   N^- f(x) &= \nu(x)\wedg (\nu(x)\lctr f(x)),\qquad x\in\Sigma.
 \end{align*} 
Theese yield projection operators onto the subspaces of multivector fields $f$ which are tangential and normal pointwise all over $\Sigma$ respectively.
To explain this algebra, construct at a given point $x\in\Sigma$ an ON-basis with $\nu(x)$ as first basis vector.
Then write a given multivector $w$ in the induced basis for $\wedge\C^n$.
It is normal if all non-zero terms contain a  factor $\nu$. It is tangential if no non-zero terms contain a factor $\nu$.
When computing $\nu\wedg w$, a factor $\nu$ is added to all tangential terms, and all normal terms are nulled.
When computing $\nu\lctr w$, a factor $\nu$ is removed from all normal terms, and all tangential terms are nulled.
The reflection operator $N= N^+-N^-$, which reflects normal multivectors across tangential multivectors, can be written
\begin{equation}  \label{eq:Ndefn}
  Nf(x) = \nu(x)\inv{f(x)}\nu(x),\qquad x\in\Sigma.
\end{equation}
The expression \eqref{eq:Ndefn}  is readily seen to be correct by writing $f(x)$ in an ON-basis adapted to $\nu(x)$,
and (anti-)commute one of the factors $\nu$ through $\inv f$.

The boundary value problems we aim at, use $N$ for the description of the boundary conditions. We now turn to the differential equation
in the domains, which generalizes the Cauchy--Riemann system:
\begin{equation}   \label{eq:Dirackeq}
   \dirac f(x)= ikf(x),
\end{equation}
where $\dirac f= \sum_{j=1}^n e_j (\partial_j f(x))$ is a Dirac operator.
The wave number $k\in\C$  is always assumed to satisfy
$$
  \im k\ge 0,
 $$ 
 where our main interest $k\in\R\setminus \{0\}$ corresponds to undampened time-harmonic waves.
From the factorization $(\dirac+ik)(\dirac-ik)= \Delta+k^2$ of the Helmholtz operator we see that when $f$ solves $\dirac f=ikf$, then each coordinate function solves the
Helmholtz equation.
Similar to the Cauchy integral formula for analytic functions, we have a reproducing formula for solutions to $\dirac f=ikf$. 
For this we need fundamental solutions: For the Helmholtz operator $\Delta+k^2$ we use the fundamental solution
$$
  \Phi_k(x)= -\frac i4 \Big( \frac k{2\pi |x|} \Big)^{n/2-1} H^{(1)}_{n/2-1}(k|x|)
$$
where $H^{(1)}_\alpha(z)$ denote the $\alpha$ order Hankel function of the first kind. 
From this we derive a fundamental solution
$\Psi_k(x)= (\dirac-ik)\Phi_k(x)$ for $\dirac +ik$. In three dimensions, which is our main interest, we have
$
  \Phi_k(x)= -\frac{e^{ik|x|}}{4\pi|x|}
$
and
$$
  \Psi_k(x)= \Big( - \frac{x}{|x|^2}+ik\big( \tfrac x{|x|}-1\big)\Big)\Phi_k(x).
$$ 
In $\R^n$, the Cauchy singular integral operator on $\Sigma$ is
\begin{equation}   \label{eq:CauchyCliffop}
   E_k f(x)=2 \pv \int_{\Sigma} \Psi_k(y-x) n(y)f(y) dy, \qquad x\in\partial\Omega.
\end{equation}
Note that the kernel uses two Clifford products.
Analogous to the static classical two dimensional case, this operator $E_k$ is a reflection operator.
 The spectral projection $E_k^+= \tfrac 12(I+ E_k)$ projects onto the subspace consisting of traces of 
 solutions to $\dirac f=ikf$ in $\Omega^+$.  
 The spectral projection $E_k^-= \tfrac 12(I- E_k)$ projects onto the subspace consisting of traces of 
 solutions to $\dirac f=ikf$ in $\Omega^-$ which satisfy a Dirac radiation condition.
 See \cite{Rspin1}.

After these definitions, we formulate boundary value problems in this Dirac framework. 
We can restrict the projections $N^+$ and $N^-$ to one of the two subspaces $E_k^+L_2$ and $E_k^-L_2$. This 
gives the four maps
\begin{align}
   N^+&: E_k^+L_2 \to N^+L_2,   \label{map:inttan} \\
      N^-&: E_k^+L_2 \to N^-L_2,   \label{map:intnor} \\
         N^+&: E_k^-L_2 \to N^+L_2, \label{map:exttan} \\
            N^-&: E_k^-L_2 \to N^-L_2.   \label{map:extnor}
\end{align}
These represents four different boundary value problems. In \eqref{map:inttan} we look for a solution in $\Omega^+$
with a prescribed tangential part on $\Sigma$.
In \eqref{map:intnor} we look for a solution in $\Omega^+$
with a prescribed normal part on $\Sigma$.
In \eqref{map:exttan} we look for a solution in $\Omega^-$
with a prescribed tangential part on $\Sigma$.
In \eqref{map:extnor} we look for a solution in $\Omega^-$
with a prescribed normal part on $\Sigma$.

Conversely, we can restrict the projections $E_k^+$ and $E_k^-$ to one of the two subspaces $N^+L_2$ and $N^-L_2$. This 
gives the four maps
\begin{align}
   E_k^+&: N^+L_2 \to E_k^+L_2,   \label{map:tanint}\\   
      E_k^+&: N^-L_2 \to E_k^+L_2,   \label{map:norint}\\   
         E_k^-&: N^+L_2 \to E_k^-L_2,  \label{map:tanext}\\  
            E_k^-&: N^-L_2 \to E_k^-L_2. \label{map:norext}  
\end{align}

These do not represent boundary value problems, but rather are ansatzes for (traces of) solutions to the Dirac equation in the domains. 
They can be combined with the above boundary 
value maps respectively to yield integral equations similar to the classical double layer potential equation.

\begin{ex}    \label{ex:planeDirac}
We describe in two dimensions $n=2$, how the Dirac framework in this section is related to the complex analysis used in the Introduction.
For complex analysis we do not regard the imaginary unit algebraically as $i=\sqrt {-1}$ but rather geometrically as the unit bivector
$j= e_1\wedg e_2$.  Note that with Clifford algebra $j^2= e_1e_2e_1e_2= -1$ and that reversion \eqref{eq:reversion} 
gives complex conjugation with respect to $j$.
In this way, a real multivector can be viewed as a pair of complex numbers $(z,w)$, where $z=z_1+jz_2\in \wedge^0\R^2\oplus \wedge^2\R^2$
and $e_1w= e_1(w_1+jw_2)= w_1e_1+ w_2e_2\in \wedge^1\R^2$.
More generally, a complex multivector corresponds to a pair of bicomplex numbers $(z,w)$ where the components $z_1,z_2,w_1, w_2$
are complex numbers with the imaginary unit $i$. This is only needed when considering differential equations with non-zero 
wave number $k\ne 0$, like for Maxwell below.
For real multivectors and $k=0$ as in the introduction, we note that the Dirac operator acts as
$$
  \dirac (f+e_1g)= \partial g+ e_1 (\rev\partial f),
$$
on a pair of complex valued functions $f$ and $g$, representing a multivector field $f+e_1g: \Omega\to \wedge\R^2$.
Here $\partial= \partial_1-j\partial_2$ and $\rev\partial= \partial_1+j\partial_2$.
Therefore a solution to $\dirac (f+e_1g)=0$ is a complex analytic function $f$ and an anti-analytic function $g$,
and the Cauchy singular integral operator $E=E_0$ from \eqref{eq:CauchyCliffop} acts as the classical Cauchy integral \eqref{eq:classyCauchy} on $f$,
with $i$ replaced by $j$, and as its anti-analytic analogue on $g$.
For the boundary conditions, we note that
$$
  N(f+e_1g)= \rev f+ e_1(-\tilde \nu^2\rev g),
$$
  where $\tilde \nu= e_1\nu$ is the unit normal vector $\nu$ represented as a complex number.
  Thus we see that $N^+(f+e_1g)$ gives the real part of $f$ and the tangential part of $g$, whereas 
   $N^-(f+e_1g)$ gives the imaginary part of $f$ and the normal part of $g$, on the boundary $\partial\Omega$.
\end{ex}

\begin{ex}   \label{ex:maxwell}
We are now in three dimensions $n=3$.
Consider Maxwell's equations in $\Omega^-$ consisting of a uniform, isotropic and conducting material, so that electric permittivity $\epsilon$,
magnetic permeability $\mu$ and conductivity $\sigma$ are constant and scalar.
We study electromagnetic wave propagation in $\Omega^-$, with material
constants $\epsilon, \mu,\sigma$, and $\Omega^+$ is assumed to be a perfect conductor.
Maxwell's equations in $\Omega^-$ then take the form 
\begin{equation}  \label{eq:maxwell}
  \begin{cases}
     \nabla \cdot  H=0, \\
     \nabla \times E= ik H, \\
     \nabla\times H=-ikE, \\
     \nabla\cdot E=0,
  \end{cases}
\end{equation}
for the (rescaled) electric and magnetic fields $E$ and $H$, where the wave number $k$ satisfies
$k^2= (\epsilon+ i\sigma/\omega)\mu \omega^2$.
The data in the scattering problem we seek to solve are incoming electric and magnetic fields $E^0$ and $H^0$
solving \eqref{eq:maxwell} in $\Omega^-$.
Our problem is to solve for the scattered electric and magnetic fields $E$ and $H$ solving \eqref{eq:maxwell} in $\Omega^-$, an inhomogeneous boundary condition at $\Sigma$ and the Silver--M\"uller radiation condition at infinity.
At $\Sigma$, since the fields vanish
in the perfect conductor $\Omega^+$, we have boundary conditions
\begin{equation}     \label{eq:Maxbcond}
  \begin{cases}
     \nu \cdot  (H+H^0)=0, \\
     \nu \times (E+E^0)= 0, \\
     \nu\times (H+H^0)=J_s, \\
     \nu\cdot (E+E^0)=\rho_s,
  \end{cases}
\end{equation}
with incoming fields $E^0$ and $H^0$ from $\Omega^-$ and surface charges and currents $\rho_s$ and $J_s$.

In $\R^3$, a multivector $w$ can be viewed a collection of two scalars $\alpha$ and $\beta$, and two vectors $a$ and $b$,
where $w= \alpha+a+*b+*\beta\in\wedge\C^3$. Here $*$ denotes the Hodge star defined by $*1=e_1\wedg e_2\wedg e_3$, $*e_1=e_2\wedg e_2$, $*e_2=-e_1\wedg e_3$ and $*e_3=e_1\wedg e_2$, which identifies $\wedge^0\C^3$ and
$\wedge^3\C^3$, and $\wedge^1\C^3$ and $\wedge^2\C^3$ respectively.
Following the setup in \cite{Rspin1}, we write the full electromagnetic field as the multivector field 
$F= E+*H$, and note that Maxwell's equations implies the Dirac equation \ref{eq:Dirackeq} for this $F$.
However, $F$ is not a general solution to \ref{eq:Dirackeq}, but satisfies the constraint that the $\wedge^0\C^3$ and $\wedge^3\C^3$ parts of $F$ vanishes.

The Maxwell boundary conditions \ref{eq:Maxbcond} means that $\nu\times E$ and $\nu \cdot H$ are prescribed on $\Sigma$.
In terms of $F$, this means that $N^+ F$ is prescribed. In \cite[Sec. 5]{Rspin1}, it  was shown that the constraint $\nabla_T\times E^0_T=ikH_N^0$,
which the incoming fields will satisfy, will imply that the Dirac solution $F$ to the boundary value problem
$$
\begin{cases}
  \dirac F= ikF & \text{in } \Omega^-, \\
  N^+F=g &\text{on } \Sigma,
\end{cases}
$$
with $g= -N^+(E^0+*H^0)|_\Sigma$, indeed is a Maxwell field in the sense that the $F=E+*H$ for two vector fields $E$ and $H$ solving
the Maxwell boundary value problem.

\end{ex}

\begin{ex}   \label{ex:layerembed}
Consider $k=0$ and the reflection operators $E=E_0$ and $N$ which we use to encode Dirac boundary conditions.
In the Introduction, we saw in two dimensions $n=2$, that the double layer potential $K$ equals the compression of $E$ to 
the subspace $N^+L_2$. More precisely, following Example~\ref{ex:planeDirac}, we mean that $E$ and $N$ are restricted to the subspace 
of complex valued, that is $\wedge^0\R^2\oplus \wedge^2\R^2$-valued, functions, where $E$ acts by the classical Cauchy singular integral and $N$ acts
by complex conjugation.
We now explain how $K$ also fits into the Dirac framework in this section, in dimensions $n\ge 3$.
With the projections $N^+$ and $N^-$, it is not possible to compress $E$ to $\wedge^0\C^n$ when $n\ge 3$.
Nevertheless, both $K$ and its adjoint 
$$
  K^*f(x)= 2 \pv \int_{\partial\Omega} \nu(x)\cdot(\nabla\Phi)(x-y) f(y) d\sigma(y), \qquad x\in\Omega,
$$
appear in different invariant subspaces for the operators $N^\pm E N^\pm$. For $K^*$, we note that when $f$ is scalar, so that $\nu f$ is a normal vector field,
then 
$$
  N^- E (\nu f)= \nu(-K^* f).
$$
So the subspace of normal vector fields is invariant under $N^-EN^-$, and its action there is given by $-K^*$ upon identifying scalars 
and normal vector fields.

For $K$, we compute for a scalar function $f\in H^1(\Sigma)$, that
$$
   N^+E (\nabla_T f)= \nabla_T(Kf),
$$
where $\nabla_T$ denotes tangential gradient.
So the subspace of tangential gradient vector fields is invariant under $N^+EN^+$, and its action there is given by $K$ on the potential.
\end{ex}

\section{Known well posedness results}   \label{sec:NEwp}

In this section we survey the known invertibility results from \cite{McMi, AGHMc, Ax} for the maps \eqref{map:inttan} - \eqref{map:norext}
on the space $L_2(\Sigma;\wedge\C^n)$ on bounded strongly Lipschitz surfaces in $\R^n$.
We include the proofs since they serve as background later in Section~\ref{sec:spin}.

\begin{prop}   \label{prop:rellichest}
  The maps  \eqref{map:inttan}- \eqref{map:extnor} all have closed range and a finite dimensional null space.
\end{prop}

\begin{proof}
   We demonstrate this for the map  \eqref{map:exttan}; the proofs for the other three are similar.
   For the proof we need the Riesz formula
   $$  
     2 f\wedg \nu= f\nu+\nu\inv f,
   $$
   valid for any $f\in\wedge\C^n$ and vector $\nu\in\wedge^1\C^n$.
  This is the reversed version of \eqref{eq:rieszwedg}.
   The strong Lipschitz and compactness assumptions on $\Sigma$ shows the existence of a smooth and compactly supported 
   vector field $\theta$ such that $\inf_\Sigma(\theta,\nu)>0$. Using this we calculate for $f\in E^-_k L_2$ that
   \begin{multline*}
      \int_\Sigma |f|^2 d\sigma\approx  \int_\Sigma |f|^2 (\theta,\nu) d\sigma= \frac 12\int_\Sigma (f(\theta\nu+\nu\theta),f)d\sigma \\
      =\re\int_\Sigma(f\nu,f\theta)d\sigma= \re\int_\Sigma (2f\wedg \nu-\nu\inv{f},f\theta) d\sigma \\
      = \re\int_\Sigma (2f\wedg \nu,f\theta) d\sigma +\int_{\Omega^-}\Big( (\inv{ikf},f\theta)+ (\inv f,(ikf)\theta)+(\inv f,\sum_j e_j f\partial_j\theta)\Big)dx.
   \end{multline*}
   The identity $\theta\nu+\nu\theta=2(\theta,\nu)$ is a special case of \eqref{eq:rieszlctr}. 
      The last identity above is an application of Stokes' theorem.
      This leads to the norm estimate
      $$
        \|f\|^2\lesssim \|f\wedg \nu\| \|f\|+ \|f\|^2_{L_2(U)},
      $$
      where $U=\Omega^-\cap\supp\theta$.
      This proves the claim since $\|f\wedg \nu\|=\|N^+f\|$ and since the Cauchy integral acts as a compact operator 
      $L_2(\Sigma)\to L_2(U)$.
\end{proof}

\begin{cor}
  The maps  \eqref{map:tanint}- \eqref{map:norext} all have closed range and a finite dimensional null space.
\end{cor}

\begin{proof}
    We show how the lower estimate
    $$
       \|f\|\lesssim \|N^+f\|+ \|f\|_{L_2(U)},\qquad f\in E_k^-L_2,
    $$
    obtained above for the map \eqref{map:exttan}, implies a similar lower bound for the map \eqref{map:norint}.
    Note that \eqref{map:exttan} and \eqref{map:norint} have the same null space $N^-L_2\cap E_k^-L_2$.
    Similar arguments for the other three pairs of maps are possible.
    
    Assume $g\in N^-L_2$ and apply the above lower bound to $f:= g-E_k^+g=E_k^-g$ to obtain
    $$
      \|g-E_k^+g\|\lesssim \|N^+(g-E_k^+g)\| + \|g-E_k^+g\|_{L_2(U)}.
    $$
    Since $N^+g=0$, this implies the claimed lower bound
    $\|g\|\lesssim \|E_k^+g\|$ modulo a compact term.
\end{proof}

\begin{thm}   \label{thm:ENFredwell}
  The maps  \eqref{map:inttan}- \eqref{map:norext} all are Fredholm maps with index zero for any $\im k\ge 0$.
\end{thm}

\begin{proof}
  By the method of continuity it suffices to consider the case $k=0$, since $k\mapsto E_k$ is a continuous map. It 
  is in fact an operator-valued analytic map. Writing $E=E_0$, we note the dualities
$$
    (Nf, \nu g) = - (f,\nu (Ng))
$$ 
and 
$$
  (Ef, \nu g) = -(f,\nu(Eg)),
$$
for any $f,g\in L_2(\Sigma;\wedge\C^n)$.
Now consider two of the restricted projections which have the same null space, for example \eqref{map:exttan} and \eqref{map:norint}.
Computing for $f\in E^-L_2$ and $g\in N^-L_2$ that
$$
  (N^+ f, \nu g)= (f,\nu g)= (f,\nu (E^+g)),
$$
it follows that the map dual to \eqref{map:exttan} is similar to \eqref{map:norint}.
Since \eqref{map:exttan} and \eqref{map:norint} have the same finite dimensional null space, their index must be zero.
\end{proof}

\begin{prop}   \label{prop:Ninj}
  The maps  \eqref{map:inttan}- \eqref{map:norext} all are isomorphisms when $\im k> 0$.
  The maps \eqref{map:exttan}, \eqref{map:extnor}, \eqref{map:tanint} and \eqref{map:norint} also are isomorphisms when
  $k\in \R\setminus\{0\}$, provided that $\Omega^-$ is a connected domain.
\end{prop}

\begin{proof}
    For a solution $f$ to $\dirac f=ikf$  in $\Omega^+$, we apply Stokes' theorem and obtain
    $$
      \int_{\Sigma} (f,\nu f) d\sigma= \int_{\Omega^+}\Big( (f,ikf)+ (ik f,f) \Big) dx= -\im k\int_{\Omega^+} |f|^2 dx.
    $$
    If $f$ belongs to $E_k^+L_2\cap N^+L_2$ or  $E_k^+L_2\cap N^-L_2$, then $(f,\nu f)=0$. If $\im k>0$, this forces $f=0$.
    
    For a solution $f$ to $\dirac f=ikf$ in $\Omega^-$, we apply Stokes' theorem and obtain
    $$
     \int_{|x|=R}(f,\tfrac x{|x|}f)d\sigma- \int_{\Sigma} (f,\nu f) d\sigma= -\im k\int_{\Omega^-\cap\{|x|<R\}} |f|^2 dx.
    $$
    If $f$ belongs to $E_k^-L_2\cap N^+L_2$ or  $E_k^-L_2\cap N^-L_2$, then $(f,\nu f)=0$ on $\Sigma$. 
    If $\im k>0$, this forces $f=0$. since in this case $f$ decays exponentially at $\infty$.
    When $k\in \R\setminus\{0\}$ we instead conclude that $ \int_{|x|=R}(f,\tfrac x{|x|}f)d\sigma=0$ for all large $R$.
    Next we note the identity  
     $$
     2|f|^2= |(\tfrac x{|x|}-1)f|^2+ 2(\tfrac x{|x|}f,f).
   $$
   Integrating this over the sphere $|x|=R$, we obtain $\lim_{R\to\infty}\int_{|r|=R}|f|^2d\sigma=0$
   using the Dirac radiation condition satisfied by $f\in E_k^- L_2$ for the term $(\tfrac x{|x|}-1)f$.
   Rellich's lemma shows that $f=0$ since $\Omega^-$ is connected and $\Delta f+k^2 f=0$.
   See \cite{McMi, Rspin1} for more details.
\end{proof}

\section{The spin ansatz and new integral equations}   \label{sec:spin}

In this section we construct new ansatzes for the boundary value problems \eqref{map:inttan}- \eqref{map:extnor}.
Instead of using the ansatzes \eqref{map:tanint}- \eqref{map:norext}, we replace the reflection operator $N$ by the reflection operator
$$
  Sf(x) =  \nu(x) f(x),\qquad x\in \Sigma.
$$
To explain why $E_k$ together with boundary conditions $S$ should yield boundary value problems with better solvability properties than $E_k$ and $N$,
we need to explain some operator algebra. 
In this we follow \cite{Ax1} and consider the operator algebra generated by two reflection operators $A$ and $B$, that is $A^2=B^2=I$, abstractly
on a Hilbert space $\mH$. We think of $B$ as encoding the differential equation through the abstract Hardy subspace projections
$B^\pm=\tfrac 12(I\pm B)$, and of $A$ as encoding boundary conditions through the two complementary projections $A^\pm=\tfrac 12(I\pm A)$.
The most important operators to describe the geometry between $A$ and $B$ are the {\em cosine operator}
$$
 C=\tfrac 12( AB+BA)
$$
and the {\em rotation operators}
$$
 AB \qquad\text{and}\qquad BA.
$$
Note that $(BA)^{-1}=AB$.

\begin{prop}    \label{prop:opinvequiv}
  For two given reflection operators $A$ and $B$, the following are equivalent.
  \begin{itemize}
     \item[ {\rm (i)}] The four restricted projections $A^+: B^+\mH\to A^+\mH$, $A^+: B^-\mH\to A^+\mH$, $A^-: B^+\mH\to A^-\mH$
     and $A^-: B^-\mH\to A^-\mH$ are isomorphisms. 
      \item [{\rm (ii)}] The four compressed projections $A^+B^+: A^+\mH\to A^+\mH$, $A^-B^-: A^-\mH\to A^-\mH$, $A^+B^-: A^+\mH\to A^+\mH$ 
      and $A^-B^+: A^-\mH\to A^-\mH$ are isomorphisms. 
     \item [{\rm (iii)}] The spectrum of the rotation operator $AB$ does not contain $+1$ or $-1$.
      \item [{\rm (iv)}] The spectrum of the cosine operator $C$ does not contain $+1$ or $-1$.
  \end{itemize}
\end{prop}

Note that (iv) is symmetric under swapping $A$ and $B$, and therefore so is (i), (ii) and (iii).

\begin{proof}
We have identities
\begin{eqnarray*}
   \tfrac 12(I+AB) = A^+B^++A^-B^-,\\
   \tfrac 12(I-AB) = A^+B^-+A^-B^+,\\
   \tfrac 12(I+C)= A^+B^+A^++A^-B^-A^-,\\
   \tfrac 12(I-C)= A^+B^-A^++A^-B^+A^-,
\end{eqnarray*}
from which the equivalences ${\rm (i)}\Leftrightarrow {\rm (iii)}$ and ${\rm (ii)}\Leftrightarrow {\rm (iv)}$ follow.
The equivalence  ${\rm (iii)}\Leftrightarrow{\rm (iv)}$ follows from identities 
\begin{eqnarray*}
  (I+AB)A(I+AB)A= 2(I+C),\\
   (I-AB)A(I-AB)A= 2(I-C).
\end{eqnarray*}
\end{proof}

\begin{ex} 
  The simplest example is when $\mH=\C^2$ with
  $$
     A=\begin{bmatrix} 1 & 0 \\ 0 & -1  \end{bmatrix}\qquad\text{and} \qquad 
      B=\begin{bmatrix} \cos(2\alpha) & \sin(2\alpha) \\ \sin(2\alpha) & -\cos(2\alpha)  \end{bmatrix},
  $$
  for some $0\le\alpha\le \pi/2$.
  In this case the ranges of the four spectral projections are $A^+=\text{{\rm span\,}}(1,0)^t$, $A^-=\text{{\rm span\,}}(0,1)^t$,
  $B^+=\text{{\rm span\,}}(\cos\alpha,\sin\alpha)^t$, $B^-=\text{{\rm span\,}}(-\sin\alpha,\cos\alpha)^t$.
  We calculate
  $$
     \tfrac 12(AB+BA)=\begin{bmatrix} \cos(2\alpha) & 0 \\ 0 & \cos(2\alpha)  \end{bmatrix}\qquad\text{and} \qquad 
      BA=\begin{bmatrix} \cos(2\alpha) & -\sin(2\alpha) \\ \sin(2\alpha) & \cos(2\alpha)  \end{bmatrix},
  $$
      with spectra $\sigma(\tfrac 12(AB+BA))=\{\cos(2\alpha)\}$ and $\sigma(BA)=\{e^{i2\alpha },e^{-i2\alpha}\}$.
      The only cases when some restricted projections fail to be invertible are when $\alpha=0$ or $\alpha=\pi/2$, in which case
      these spectra hit $\{+1,-1\}$.
      The optimal geometry from the point of boundary value problems is when $\alpha=\pi/4$, in which case the spectra are $\{0\}$ and $\{+i, -i\}$ respectively.
\end{ex}

Consider now $B= E_k$ and $A=S$. In this case we note that the rotation operator $BA$ is given by
$$
  E_kSf(x) = 2\pv \int_\Sigma \Psi_k(y-x) f(y) d\sigma, \qquad x\in\Sigma,
  $$
since $\nu^2=1$.
This is really the core observation of this paper.
For $k=0$, we note that this yields a skew-symmetric operator $(ES)^*=-ES$, with purely imaginary spectrum.
In particular it stays well away from $\pm 1$, and therefore it is clear that all boundary value problems described by $E$ and $S$
are well posed.
Note that this follows from Proposiition~\ref{prop:opinvequiv} by abstract arguments and is not using the strong Lipschitz assumption 
on $\Sigma$, in contrast to well posedness for the pair $E$ and $N$ in Section~\ref{sec:NEwp}.

  For non-zero $k$, we have at least that $(E_kS)^*=-E_kS$ modulo compact operators.
In this way, we see from abstract considerations only, that  all restricted projections
\begin{align}
   S^+&: E_k^+L_2 \to S^+L_2,   \label{map:inttans} \\
      S^-&: E_k^+L_2 \to S^-L_2,   \label{map:intnors} \\
         S^+&: E_k^-L_2 \to S^+L_2, \label{map:exttans} \\
            S^-&: E_k^-L_2 \to S^-L_2,   \label{map:extnors} \\
   E_k^+&: S^+L_2 \to E_k^+L_2,   \label{map:tanints}\\   
      E_k^+&: S^-L_2 \to E_k^+L_2,   \label{map:norints}\\   
         E_k^-&: S^+L_2 \to E_k^-L_2,  \label{map:tanexts}\\  
            E_k^-&: S^-L_2 \to E_k^-L_2. \label{map:norexts}  
\end{align}
  are Fredholm operators with index zero. We summarize and complement with injectivity results in the following proposition.
 
\begin{prop}   \label{prop:Sinj}
    The maps  \eqref{map:inttans}- \eqref{map:norexts} are all Fredholm operators with index zero when $\im k\ge 0$, and isomorphisms 
    when $k=0$. Moreover, the norms of these inverses (Fredholm inverses) are bounded by $2$, when $k=0$ $(\im k\ge 0)$.
  The maps \eqref{map:intnors}, \eqref{map:exttans}, \eqref{map:norints} and \eqref{map:tanexts} are  isomorphisms when
  $\im k\ge 0$.
\end{prop}

\begin{proof}
 The bounds of the (Fredholm) inverses follows from the skew-adjointness of $ES$ by the formulas in the proof of 
 Proposition~\ref{prop:opinvequiv}.
 The fact that the index is zero follows from the method of continuity for $k\ne 0$, since $k\mapsto E_k$ is continuous.

  The injectivity results that remains to prove are that $E_k^+L_2\cap S^+L_2=\{0\}$ and  $E_k^-L_2\cap S^-L_2=\{0\}$.
 The idea of proof is similar to Proposition~\ref{prop:Ninj}.
      For a solution $f$ to $\dirac f=ikf$  in $\Omega^+$, we apply Stokes' theorem and obtain
    $$
      \int_{\Sigma} (f,\nu f) d\sigma= \int_{\Omega^+}\Big( (f,ikf)+ (ik f,f) \Big) dx= -\im k\int_{\Omega^+} |f|^2 dx.
    $$
    If $f\in E_k^+L_2\cap S^+L_2$, then $(f,\nu f)=|f|^2\ge 0$. If $\im k\ge 0$, this forces $f=0$ on $\Sigma$ and therefore
    by the Cauchy formula also $f=0$ in $\Omega^+$.
    
    For $f\in E_k^-L_2\cap S^-L_2$ we have $(f,\nu f)=-|f|^2$ on $\Sigma$, and a similar application of Stokes' theorem yields
    $$
     \int_{|x|=R}(f,\tfrac x{|x|}f)d\sigma+ \int_{\Sigma} |f|^2 d\sigma= -\im k\int_{\Omega^-\cap\{|x|<R\}} |f|^2 dx.
    $$ 
    If $\im k>0$, this forces $f=0$ as before, letting $R\to\infty$. since in this case $f$ decays exponentially at $\infty$.
    Also when $k=0$, $f$ has enough decay for us to conclude. (However, the case $k=0$ is already taken care of by the skew-adjointness
    of $ES$.)
    When $k\in \R\setminus\{0\}$ we instead conclude that $ \int_{|x|=R}(f,\tfrac x{|x|}f)d\sigma\le 0$ for all large $R$.
       Integrating the identity  
     $$
     2|f|^2= |(\tfrac x{|x|}-1)f|^2+ 2(\tfrac x{|x|}f,f)
   $$
 over the sphere $|x|=R$, we obtain $\lim_{R\to\infty}\int_{|r|=R}|f|^2d\sigma=0$ since 
   the Dirac radiation condition for $f$ at infinity shows the the first term on the right vanishes at infinity.
   As in Proposition~\ref{prop:Ninj}, this forces $f=0$ by Rellich's lemma if $\Omega^-$ is connected. If $\Omega^-$ have bounded connected
   components, we can argue as for  $E_k^+L_2\cap S^+L_2=\{0\}$ to conclude that $f=0$ also in these components.
\end{proof}

Proposition~\ref{prop:Sinj} is the main result that we need to obtain the announced spin integral equations for solving Dirac 
boundary value problems.
The idea is to use the ansatz~\ref{map:norints}, which is always invertible by Proposition~\ref{prop:Sinj},
 for the interior boundary value problems 
\eqref{map:inttan} and \eqref{map:intnor}. This leads to the integral equations
\begin{equation}   \label{eq:prespinint}
  N^\pm E_k^+ : S^-L_2\to N^\pm L_2. 
\end{equation}
Similarly, for the exterior boundary valur problems \eqref{map:exttan} and \eqref{map:extnor}, we use the ansatz
\eqref{map:tanexts},  which is also  always invertible by Proposition~\ref{prop:Sinj}. This gives integral equations
\begin{equation}   \label{eq:prespinext}
    N^\pm E_k^- : S^+L_2\to N^\pm L_2. 
\end{equation}
This is almost what we want: the integral operators \eqref{eq:prespinint} and \eqref{eq:prespinext} are invertible 
if and only if the corresponding boundary value problems \eqref{map:inttan} and \eqref{map:intnor}, or
\eqref{map:exttan} and \eqref{map:extnor} respectively, are well posed,
and both the domains and ranges are simple pointwise defined subspaces of $L_2(\Sigma)$.
We can however improve these integral equations a little further for numerical implementation, so that both the domains and ranges 
are the same subspace, and not depending on $\nu(x)$ like $S^\pm L_2$ and $N^\pm L_2$ do.
To this end, we apply again the abstract setup for boundary value problems described in this section.
Consider the reflection operator
$$
   Tf(x) = \inv{f(x)},\qquad x\in\Sigma,
 $$ 
given by pointwise involution of the multivector field.
The corresponding spectral subspaces are
$$
  T^+L_2 = L_2(\Sigma;\wedge^\ev\C^n) \qquad \text{and}\qquad 
    T^-L_2 = L_2(\Sigma;\wedge^\od\C^n).
$$
Computing the relevant cosine operators, we have
\begin{eqnarray*}
  ( NS+SN )f= \nu\inv{\nu f}\nu+\nu\nu\inv f\nu=0, \\
  TS+ST)f=\inv{\nu f}+  \nu\inv f=0,\\
  (TN+NT)f= \inv{n\inv f\nu}+\nu \inv{\inv f}\nu= 2\nu f\nu\ne 0.
\end{eqnarray*}

The proof of Proposition~\ref{prop:opinvequiv}
 give us explicitly invertible maps between subspaces $S^\pm L_2$
and subspaces $N^\pm L_2$ on the one hand, and  explicitly invertible maps between subspaces $S^\pm L_2$
and subspaces $T^\pm L_2$ on the other hand. 
Indeed, we note that if $A$ and $B$ are two reflection operators on a Hilbert space $\mH$ satisfying $AB+BA=0$,
then  the associated eight restricted projections are pairwise inverse, up to a factor $2$, as follows.
\begin{eqnarray*}
   (B^+: A^+\mH\to B^+\mH)^{-1} = 2(A^+: B^+\mH\to A^+\mH) \\
   (B^-: A^-\mH\to B^-\mH)^{-1} = 2(A^-: B^-\mH\to A^-\mH) \\
   (B^-: A^+\mH\to B^-\mH)^{-1} = 2(A^+: B^-\mH\to A^+\mH) \\
   (B^+: A^-\mH\to B^+\mH)^{-1} = 2(A^-: B^+\mH\to A^-\mH) 
\end{eqnarray*}

We can now formulate the main result of this paper, namely spin  integral equations for solving the boundary value problems for
the differential equation $\dirac f=ikf$ with prescribed tangential or normal part of the field at the boundary. 

\begin{thm}   \label{thm:main}
     Let $\Omega^+\subset\R^n$ be a bounded strongly Lipschitz domain, with exterior domain $\Omega^-$, and consider a wave number $\im k\ge 0$.
     \begin{itemize}
     \item
       The interior boundary value problem to find a solution $f$ to $\dirac f=ikf$ in $\Omega^+$ with prescribed tangential/normal part 
       $N^\pm f=g$ at $\Sigma$ is well posed in the sense that $N^\pm: E_k^+L_2\to N^\pm L_2$ is invertible, if and only if the singular integral equation
       $$
          T^+S^-N^\pm E_k^+S^- h=T^+S^- g
       $$
       is uniquely solvable for $h\in T^+L_2$. In this case the solution to the boundary value problem is $f=E_k^+ S^-h$ at $\Sigma$.
     \item
       The exterior boundary value problem to find a solution $f$ to $\dirac f=ikf$ in $\Omega^\pm$ with prescribed tangential/normal part 
       $N^\pm f=g$ at $\Sigma$ is well posed in the sense that $N^\pm: E_k^- L_2\to N^\pm L_2$ is invertible, if and only if the singular integral equation
       $$
          T^+S^+N^\pm E_k^-S^+ h=T^+S^+ g
       $$
       is uniquely solvable for $h\in T^+L_2$. In this case the solution to the boundary value problem is $f=E_k^- S^+h$ at $\Sigma$.
     \end{itemize}
\end{thm}

\begin{proof}
For the interior boundary value problems, the ansatz 
$E_k^+: S^-L_2 \to E_k^+L_2$ is an invertible map for any $\im k\ge 0$ by Proposition~\ref{prop:Sinj}.
For the exterior boundary value problems, the ansatz 
$E_k^-: S^+L_2 \to E_k^-L_2$ is an invertible map for any $\im k\ge 0$ by Proposition~\ref{prop:Sinj}.
We have also seen that $T^+S^\pm: N^\pm L_2\to T^+L_2$ and $S^\pm: T^+L_2\to S^\pm L_2$ are invertible maps.
These invertible maps enable us to fomulate the boundary value problems as singular integral equations on
the subspace $L_2(\Sigma;\wedge^\ev\C^n)$ as stated.
\end{proof}

\begin{ex}  \label{ex:2dint}
  We saw in Example~\ref{ex:planeDirac} how the Dirichlet problem for the Laplacian in $\Omega^+\subset\R^2$, or equivalently the 
  Hilbert  boundary value problem for analytic functions with prescribed real part on $\Sigma$, can be formulated in terms of invertibility 
  of $N^+: E^+L_2\to N^+L_2$.

   Theorem~\ref{thm:main} allow us to solve this boundary value problem as a real linear singular integral equation in the space $L_2(\Sigma; \C)$
   as follows.  
   Given the real valued Dirichlet datum $g\in L_2(\Sigma;\R)$, we compute $T^+S^-g=\tfrac 12 g$.
   To see that the equation $T^+S^-N^+ E^+S^- h=T^+S^- g$ reduces to \eqref{eq:2dspineq} using complex algebra, 
   we note that $T^+S^-N^+S^-h= \tfrac 14 h$ since $S$ anti-commute with $T$ and $N$.
   Writing the Cauchy integral with complex algebra, we have
   $$
     ES^-h(z)= \frac 1{2\pi j} \pv\int_\Sigma \frac {h(w)dw}{w-z} -\frac {e_1}{2\pi}\pv\int_\Sigma\frac {h(w)|dw|}{\rev w-\rev z}.
   $$
   Computing $4T^+S^-N^+(ES^-h)$, we obtain  \eqref{eq:2dspineq} with $i$ replaced by $j$.
   In the formula $f=E^+ S^-h$ for the solution to the boundary value problem, we need only to evaluate the $\wedge^\ev\R^2$
   part of $f$: The auxiliary anti-analytic function given by the $\wedge^1\R^2$ part will be trivial due to our choice of $g$.
   Thus we end up with the classical Cauchy integral of $h$ for the solution $u+jv$.
\end{ex}

\begin{ex}    \label{ex:maxwellintequ}
  We saw in Example~\ref{ex:maxwell} how the Maxwell scattering problem in $\Omega^-\subset\R^3$ against a perfect conductor $\Omega^+$
  can be formulated in terms of invertibility of $N^+: E_k^-L_2\to N^+L_2$.
   Theorem~\ref{thm:main} allows us to solve this linear equation as a singular integral equation in the space $L_2(\Sigma; \C^4)$
   as follows.
  Given the incoming electric and magnetic fields $E^0$ and $H^0$, we compute the tangential Dirac data $g= -N^+(E^0+*H^0)|_\Sigma$.
  Note that $g$ depends on the tangential part of $E^0$ and the normal part of $H^0$.
  Compute the bivector field 
  $$
     \tilde g:= T^+S^+g= - \nu\wedg E^0- (\nu, H^0)(*\nu) : \Sigma\to\wedge^2\C^3.
  $$
  It is straightforward to compute that $T^+S^+N^+E_k^-S^+h= T^+S^+g$ amounts to 
  solving the singular integral equation
  \begin{equation}  \label{eq:maxwellspineq}
    \frac 14 h(x)+   M(x)\,\pv\int_\Sigma \Psi_k(y-x) (1+ \nu(y))h(y)d\sigma(y)=2 \tilde g(x),\qquad x\in\Sigma,
  \end{equation}
  for $h\in L_2(\Sigma;\wedge^\ev\C^3)$,
  where $M$ is the multiplier 
  $$
  M=T^+S^+N^+: \alpha+ a+*b+*\beta\mapsto \alpha+\nu\wedg a+(\nu,b)(*\nu).
  $$
  The solution to the Maxwell scattering problem is then
  $$
    E(x)+*H(x)=- \int_\Sigma \Psi_k(y-x) (1+ \nu(y))h(y)d\sigma(y), \qquad x\in\Omega^-.
  $$
  This is the algorithm that we propose in this paper for solving the Maxwell scattering problem against a perfect conductor.
\end{ex}

\section{Maxwell scattering in piecewise constant media}   \label{sec:dielectric}

We formulated in Example~\ref{ex:maxwellintequ} a spin integral equation for solving the Maxwell scattering problem 
against a perfect conductor, which is a singular integral equation for four scalar functions.
In this section, we fomulate a similar spin integral equation for solving more general scattering problems, for 
time-harmonic Maxwell's equations at frequency $\omega$.
We do not aim to present a complete solvability theory in this section, since it requires a solution of fundamental open problems.
Instead we fomulate the algorithm and describe the future work that is needed.

 We denote by $N$ the number of bounded materials, and write
$$
  \R^3= \Omega_0 \cup \Omega_1\cup\ldots\cup \Omega_N\cup\Sigma
$$
disjointly.
Here $\Omega_j$ are assumed to be bounded open sets, with Lipschitz regular boundaries $\partial \Omega_j\subset\Sigma$, $j=1,\ldots, N$, 
and $\Omega_0$ is the complement of a bounded Lipschitz domain.
The Lipschitz interface $\Sigma$ is $\Sigma= \partial\Omega_1\cup\ldots\cup\partial\Omega_N$.
Write $\Sigma_{i,j}=\clos{\Omega_i}\cap \clos {\Omega_j}$.
A unit normal vector at a boundary point $x\in\Sigma$, which is well defined almost everywhere, 
is denoted $n=n(x)$. By $n_j= n_j(x)$ at $x\in\partial\Omega_j$ we mean the unit normal vector which is outward pointing relative $\Omega_j$.

The region $\Omega_j$, $j=0,1,\ldots, N$, we assume represent a
 homogeneous, linear and isotropic material, with 
electric permittivity $\epsilon_j$, magnetic permeability $\mu_j$ and conductivity $\sigma_j$ as constant and scalar quantities.
We formulate Maxwell's equations in $\Omega_j$ as 
\begin{equation}  \label{eq:maxwell}
  \begin{cases}
     \nabla \cdot  H=0, \\
     \nabla \times E= ik_j H, \\
     \nabla\times H=-ik_jE, \\
     \nabla\cdot E=0,
  \end{cases}
\end{equation}
where the wave number is $k_j=\omega\alpha_j/\beta_j$, with $\alpha_j= (\epsilon_j+i\sigma_j/\omega)^{1/2}$ and $\beta_j= \mu_j^{-1/2}$.
We choose to normalise the fields so that $E$ denotes the geometric mean of the electric field and displacement, and $H$ denotes the
 geometric mean of the magnetic field and intensity, so that the square of the fields has energy density as dimension.
 In particular this means that at the interface $\Sigma$, we have jump conditions 
 which require continuity of 
 $$
         n \cdot  (\beta^{-1}H), \quad
     n \times (\alpha^{-1}E),\quad
     n\times (\beta H),\quad\text{and}\quad
     n\cdot (\alpha E)
 $$
if Maxwell's equations  for the original electric and magnetic fields are to hold in distributional sense in all $\R^3$.

The data in the scattering problem we seek to solve are incoming electric and magnetic fields 
$E^\inc$ and $H^\inc$
solving \eqref{eq:maxwell} with $k=k_0$ in $\Omega_0$.
Our problem is to solve for electric and magnetic fields $E^j$ and $H^j$
\begin{itemize}
\item  solving Maxwell's equations \eqref{eq:maxwell} with wave number $k_j$ in $\Omega_j$,
$j=0,1,\ldots,N$, 
\item  with $E^0$, $E^0$ satisfying the Silver--M\"uller radiation condition at infinity, see \cite[eq. (4)]{Rspin1}, 
\item 
and where
$$
   E^\inc+E^0+\ldots+ E^N\quad\text{and}\quad H^\inc+H^0+\ldots+ H^N
$$
solve Maxwell's equations in distributional sense across $\Sigma$.   
\end{itemize}

We note as before that in terms of the electromagnetic multivector field $F= E+ *H$, we can write Maxwell's equations \eqref{eq:maxwell}
as the Dirac equation
$$
   \dirac F= ik_j F
$$
in $\Omega_j$,  and well posedness in $L_2=L_2(\Sigma;\wedge\C^3)$ of  the scattering problem descibed above follows from invertibility
of the map
$$
   B_\Sigma:\bigoplus_{j=0}^N E_j L_2\to L_2: (f_j)_{j=0}^n\mapsto \sum_{j=0}^N N_j f_j.
$$
Here $E_jL_2$ denotes the image of $E^+_{k_j}L_2(\partial\Omega_j)$ under the inclusion $L_2(\partial\Omega_j)\subset L_2(\Sigma)$,
where $N_j$ denotes the map
$$
  N_j f(x) = \nu_j \wedg (\beta_j^{-1}T^+f+ \alpha_j^{-1}T^- f)+  \nu_j \lctr (\beta_jT^+f+ \alpha_jT^- f), \qquad x\in\partial\Omega_j,
$$
and $N_jf(x)=0$ when $x\in\Sigma\setminus \partial\Omega_j$.
Recall that $T^+$ is projection onto $L_2(\Sigma;\wedge^\ev\C^3)$ and $T^-$ is projection onto $L_2(\Sigma;\wedge^\od\C^3)$.
We  use the spin ansatz
$$
  S_\Sigma: L_2\to \bigoplus_{j=0}^N E_j L_2: f\mapsto (E_j S_j f)_{j=0}^N
$$
where 
\begin{align*}
  S_j f(x)&= \tfrac 12(1-\nu_j(x))f(x), \\
  E_j f(x) &= E_{k_j}^+(f|_{\partial\Omega_j})(x),
\end{align*}
when $x\in\partial\Omega_j$ and $S_jf(x)= E_jf(x)=0$ when $x\in\Sigma\setminus\partial\Omega_j$.
 
 With this setup we obtain the following spin integral equation for solving the above Maxwell scattering problem.
 
\begin{alg}   \label{alg:maxscat}
Let $E^\inc$ and $H^\inc$ be the incoming rescaled electric and magnetic fields in $\Omega^0$, and define
$$
g= N_0(E^\inc|_{\partial\Omega^0}+* H^\inc|_{\partial\Omega^0})\in L_2(\Sigma;\wedge\C^3).
$$
Solve the singular integral equation 
$$
  B_\Sigma S_\Sigma h=- g
$$
for $h\in L_2(\Sigma;\wedge\C^3)$.
Then the solution to the Maxwell scattering problem in this section is given by
$$
   E^j(x)+*H^j(x)=\frac 12 \int_{\partial\Omega_j} \Psi_k(y-x) (1-\nu_j(y)) h(y) d\sigma(y),\qquad x\in\Omega_j.
$$
\end{alg}

Since $S_j$ are complementary projections in $L_2(\Sigma)$, it follows from Proposition~\ref{prop:Sinj} that the spin 
ansatz $S_\Sigma$ is an invertible map for any $\im k\ge 0$.
As we discussed in the introduction, for the spin integral equation to be computationally useful we also need to show that
the Dirac scattering problem which we embed the Maxwell scattering problem into, 
is well posed.
We conjecture that this is the case for Algorithm~\ref{alg:maxscat}. It appears though that even for one bounded material, $N=1$,
this is beyond the currently available $L_2(\Sigma)$ solvability techniques, which are based on Rellich estimates like in
Proposition~\ref{prop:rellichest}, if we allow general Lipschitz interfaces.
To show the main problem, consider the jump relations 
\begin{align}
  \nu\wedg (\beta_1^{-1} T^+f_1 - \beta_0^{-1}T^+ f_0) &= \nu\wedg T^+g,\label{eq:maxtr1} \\   
    \nu\wedg (\alpha_1^{-1} T^-f_1 - \alpha_0^{-1}T^- f_0) &= \nu\wedg T^- g,\label{eq:maxtr2} \\
        \nu\lctr (\beta_1 T^+f_1 - \beta_0T^+ f_0) &= \nu\lctr T^+ g, \label{eq:maxtr3}  \\
        \nu\lctr (\alpha_1 T^-f_1 - \alpha_0T^- f_0) &= \nu\lctr T^-g, \label{eq:maxtr4}
\end{align}
on $\Sigma = \partial\Omega_1=\partial\Omega_0$ when $N=1$.
We want to show that $(f_1,f_0)\mapsto g$ is a Fredholm map in the $L_2(\Sigma)$ topology.
To this end we note that 
$$
  T^- f_j= T^- E_{k_j}^{(-1)^j} f_j= T^- E^{(-1)^j} f_j +Kf_j=  E^{(-1)^j} T^- f_j +Kf_j, \qquad j=0,1,
$$
where $E=E_0$ and $K=(-1)^j\tfrac 12 T^-(E_{k_j}-E)$ is a compact operator.
Adding $\alpha^+\alpha^-$ times \eqref{eq:maxtr2} and  \eqref{eq:maxtr4} 
yields the estimate
$$
  \|(\lambda I -EN)T^-(f_1+f_0)) \lesssim \|T^- g\|+ \|K'T^-(f_1+f_0)\|,
$$
with a compact operator $K'$.
This yields a Fredholm bound on $\|T^-f_1\|+  \|T^-f_0\|$, provided $\lambda= (\alpha^++\alpha^-)/(\alpha^+-\alpha^-)$
is outside the Fredholm spectrum of the rotation operator $EN$.
Similarly, by adding $\beta^+\beta^-$ times \eqref{eq:maxtr1} and  \eqref{eq:maxtr3}, we obtain a Fredholm
bound on  $\|T^+f_1\|+  \|T^+f_0\|$, provided $(\beta^++\beta^-)/(\beta^+-\beta^-)$
is outside the Fredholm spectrum of the rotation operator $EN$.
As shown in \cite{Ax1}, these conditions are equivalent to that $(\alpha^+/\alpha^-)^2$ and $(\beta^+/\beta^-)^2$ maps
into the Fredholm resolvent set for the cosine operator 
$$
   \tfrac 12(EN+NE)
$$
by the map $z\mapsto (z+1)/(z-1)$.
Thus, if we allow arbitrary conductivity $\sigma\ge 0$ and permittivity $\epsilon>0$, we need to know that the spectral 
radius of the cosine operator is $\le 1$.
As we have seen in Example~\ref{ex:layerembed} and Proposition~\ref{prop:opinvequiv}, the classical double layer potential 
operator embeds into this cosine operator, so in particular we need to know that the spectral radius of $K$ is at most one.
This well known spectral radius conjecture is to the authors knowledge still an open problem for general Lipschitz surfaces.

However, if furthermore $\Sigma$ is smooth, then the spectrum of $(EN+NE)/2$ is contained in the unit disk. Indeed
$$
  E^*f(x)-Ef(x)= 2\pv \int_\Sigma (\Psi(y-x)\nu(y)+ \nu(x)\Psi(y-x))f(y)\sigma(y)\approx 2 Kf(x),
$$
modulo compact operators is $\nu$ is smooth, since $ab+ba=2(a,b)$ with Clifford algebra.
Here the double layer potential $K$ acts componentwise on the multivector field $f$.
Since $E$ is a reflection operator it follows that $E$ is a compact perturbation of a unitary operator.
In particular we deduce that the Fredholm spectrum of the cosine operator is contained in $[-1,1]$.

To show injectivity of $B_\Sigma$ one can generalise the methods in Proposition~\ref{prop:Ninj}.
We omit the details and refer to \cite{Ax3}.

\section{Problems with the classical ansatz}  \label{sec:2d}

We end this paper with an explicit computation on a conical domain, elaborating on Mellin transform techniques of Fabes, Jodeit and Lewis~\cite{FJL},
that shows that the classical double layer potential equation
may have a condition number which is significantly worse than that of the underlying boundary value problem.
This in contrast to the spin integral equations proposed in Theorem~\ref{thm:main}, which typically is no
worse than the boundary value problem numerically.
By localising the result below, we obtain similar results for bounded domains which have corners.

Consider the double layer potential operator $K$ in dimension $2$ given by \eqref{eq:K2D} in the Introduction,
on a cone 
$$
   \Omega= \sett{z\in\C}{0<\arg z<\theta}.
$$
We are in particular interested in the limit as $\theta\to 0^+$.
Recall that $K$ is the composition of the restricted Cauchy projection \eqref{eqfactor1} and the restricted real part projection \eqref{eqfactor2},
and that all these operators are considered as real linear only.
The purpose of this example is to show by explicit computation, that
\begin{align*}
   \|(I+K)^{-1}\| &\approx 1/\theta^2,\\
   \|(N^+:E^+L_2\to N^+L_2)^{-1}\| & \approx 1/\theta,\quad\text{and}\\
   \|(E^+:N^+L_2\to E^+L_2)^{-1}\| & \approx 1/\theta
\end{align*}
as $\theta\to 0^+$.
This means that the interior Dirichlet problem which we intend to solve become increasingly illposed at the rate $1/\theta$.
On the other hand, the operator $I+K$ which is classically used to solve the Dirichlet problem become ill-conditioned at the faster 
rate $1/\theta^2$.
Note that the operators $K$, $E$ and $N$ are themselves uniformly bounded as $\theta\to 0^+$.

Since the computation uses the Fourier transform, it is natural to complexify these real linear operators, which we do as follows.
The imaginary unit $i$ used in the definition of $E$ and $N$, we write as $j$ and rather think of as the unit bivector $j=e_1\wedg e_2$ 
as in Example~\ref{ex:planeDirac},
which squares to $j^2=-1$ with the Clifford product. In the framework from Section~\ref{sec:higheralg},
this means that $N$ is reversion (rather than complex conjugation) and $N^+$ is projection onto scalars $\wedge^0\C^2\approx \C$.
Our operators act on functions taking values in the even subalgebra
$$
  \wedge^\ev\C^2= \sett{z+jw}{z,w\in\C},
$$
that is the commutative algebra of bicomplex numbers.

The first step is to apply the Mellin transform to $E$, meaning that we first pull back $E$ by the isometry
$$
  \gamma^*: L_2(\partial\Omega;\wedge^\ev\C^2)\to L_2(\R;(\wedge^\ev\C^2)^2):
  f(x)\mapsto \begin{bmatrix} e^{t/2}f(e^t) \\ e^{t/2}f(e^{t+j\theta}) \end{bmatrix},
$$
followed by the componentwise Fourier transform $\mF  f(\xi)= \int_\R f(t)e^{-i\xi t}dt$.
These computations, which involve some residue calculus using the imaginary unit $j$, lead to the formula
$$
  \begin{bmatrix}
     -ij\tanh(\pi\xi)  & \frac{e^{-j\alpha/2}}{\cosh(\pi\xi)}(j\cosh(\alpha\xi)-i\sinh(\alpha\xi)) \\
      \frac{e^{-j\alpha/2}}{\cosh(\pi\xi)}(-j\cosh(\alpha\xi)-i\sinh(\alpha\xi)) &  ij\tanh(\pi\xi)
  \end{bmatrix}
$$
for $  \mF \gamma^* E (\gamma^*)^{-1}\mF^{-1}$, where $\alpha:= \pi-\theta$.
From this we compute
$$
  N^+EN^+\approx
   \frac{\sin(\alpha/2-i\alpha\xi)}{\cosh(\pi\xi)} 
    \begin{bmatrix}
       0 & 1 \\ 1 & 0 
     \end{bmatrix},
$$
from which $ \|(I+K)^{-1}\|\approx 1/\theta^2$ follows.
To obtain the claimed $1/\theta$ bound on the inverses of the restricted projections, we note 
from the identity $\tfrac 12N(I+EN)N= N^+E^++N^-E^-$, the duality between 
$E^+:N^+L_2\to E^+L_2$ and $N^-: E^-L_2\to N^-L_2$, up to similarity as in Theorem~\ref{thm:ENFredwell},
and the uniform boundedness of $E$ and $N$ that it suffices to prove that 
$$
  \|(I+EN)^{-1}\|\approx 1/\theta.
$$
To this end, we use that the inverse of a matrix $A = \begin{bmatrix} a & b \\ c & d  \end{bmatrix}$, where $a,b,c,d$ in
general are non-commuting operators, is given by
\begin{equation}    \label{eq:noncomminv}
  A^{-1} =  \begin{bmatrix}
     (d-ca^{-1}b)^{-1} &  - (d-ca^{-1}b)^{-1}ca^{-1} \\
     -a^{-1}b  (d-ca^{-1}b)^{-1} & a^{-1}+a^{-1}b (d-ca^{-1}b)^{-1} ca^{-1}
  \end{bmatrix}.
\end{equation}
Applying this to $A$ equal to the Fourier multiplier of $I+EN$, it suffices for us to bound $a^{-1}$ and $ (d-ca^{-1}b)^{-1}$.
We calculate 
$$
  a^{-1}= (1-ij\tanh(\pi\xi) N)^{-1} = \frac {1+ij\tanh(\pi\xi)N}{1+\tanh^2(\pi\xi)},
$$
which is bounded and independent of $\theta$.
With straightforward bicomplex algebra, and noting that reversion $N$ commutes with $i$ but anti-commutes with $j$,  we also compute
$$
   d-ca^{-1}b= 1+ij\tanh(2\pi\xi) N +  (\cos(\alpha)+j\sin\alpha)\frac {\cosh(2\alpha\xi)-ij\sinh(2\alpha\xi)}{\cosh(2\pi\xi)}.
$$
To bound the inverse, we write matrices in the complex basis $\{1,j\}$ as $j= \begin{bmatrix} 0 & 1 \\ -1 & 0  \end{bmatrix}$
and $N=  \begin{bmatrix} 1 & 0 \\ 0 & -1  \end{bmatrix}$. 
From the standard commutative version of \eqref{eq:noncomminv}, we obtain
$$
  (d-ca^{-1}b)^{-1} =\frac 1{2+ 2X} \begin{bmatrix}
      1+ X & 
        i\tanh(2\pi\xi)+Y  \\
           i\tanh(2\pi\xi)-Y
     & 1+X
  \end{bmatrix},
$$
where
\begin{eqnarray*}
  X= \frac{\cos\alpha\cosh(2\alpha\xi)+i\sin\alpha\sinh(2\alpha\xi)}{\cosh(2\pi\xi)},\\
  Y=  \frac{i \cos\alpha\sinh(2\alpha\xi)-\sin\alpha\cosh(2\alpha\xi)}{\cosh(2\pi\xi)}.
\end{eqnarray*}
Doing the estimates, this yields $|(d-ca^{-1}b)^{-1}|\approx 1/\theta$, from which we deduce that the norms of the inverses of the
 restricted projections \eqref{eqfactor1} and  \eqref{eqfactor2}  are of the order $1/\theta$ in this example.

\bibliographystyle{acm}

\end{document}